\documentclass[11pt]{amsart}
\usepackage{latexsym,amssymb}
\usepackage{amsmath}
\usepackage{amscd}
\usepackage{xypic}

\newcommand{\llar}{-\kern-5pt-\kern-5pt\longrightarrow}

\newtheorem{Theorem}{Theorem}[section]

\newtheorem{thm}[Theorem]{Theorem}
\newtheorem{Lemma}[Theorem]{Lemma}
\newtheorem{Corollary}[Theorem]{Corollary}
\newtheorem{corollary}[Theorem]{Corollary}

\newtheorem{Proposition}[Theorem]{Proposition}

\newtheorem{conj}[Theorem]{Conjecture}

\newtheorem{Remark}[Theorem]{Remark}

\newtheorem{Example}[Theorem]{Example}

\newtheorem{Definition}[Theorem]{Definition}
\newtheorem{definition}[Theorem]{Definition}

\theoremstyle{definition}

\theoremstyle{plain}
\theoremstyle{definition}

\theoremstyle{remark}

\newcommand{\rmH}{\mathrm{H}}
\newcommand{\rmK}{\mathrm{K}}

\def\Ass{\mbox{\rm Ass}}
\def\Assh{\mbox{\rm Assh}}

\def\depth{\mbox{\rm depth}}
\def\ds{\displaystyle}

\def\Supp{\mbox{\rm Supp}}

\def\Hom{\mbox{\rm Hom}}

\def\im{\mbox{\rm Im}}

\def\l{\lambda}
\def\lar{\longrightarrow}

\def\m{\mathfrak{m}}

\def\n{\mathfrak{n}}

\def\p{\mathfrak{p}}

\def\q{\mathfrak{q}}

\def\QED{\hfill$\Box$}

\def\rar{\rightarrow}

\def\xx{{\bf x}}
\def\hh{{\bf h}}

\def\RR{{\bf R}}
\def\SS{{\bf S}}

\def\XX{{\bf X}}
\def\xx{{\bf x}}
\def\ff{{\bf f}}

\def\RR{{\bf R}}
\def\TT{{\bf T}}

\def\m{{\mathfrak m}}
\def\p{{\mathfrak p}}
\def\n{\mathfrak n}
\def\H{{\mathrm H}}

\def\Ass{\mbox{\rm Ass}}
\def\demo{\noindent{\bf Proof. }}
\def\QED{\hfill$\Box$}
\def\Deg{\mbox{\rm Deg}}
\def\hdeg{\mbox{\rm hdeg}}

\def\Hom{\mbox{\rm Hom}}

\def\Tor{\mbox{\rm Tor}}

\newcommand{\Rees}{\mbox{${\mathcal R}$}}
\def\gr{\mbox{\rm gr}}
\def\br{\mbox{\rm br}}
\def\bar#1{{\overline{#1}}}

\newcommand{\fkm}{\mathfrak{m}}

\begin{document}

\title{The Chern Numbers and Euler Characteristics of Modules}

\author{L. Ghezzi} \address{Department of Mathematics, New York City
College of Technology-Cuny, 300 Jay Street, Brooklyn, NY 11201, U. S.
A.} \email{lghezzi@citytech.cuny.edu}
\author{S. Goto}
\address{Department of Mathematics, School of Science and Technology,
Meiji University, 1-1-1 Higashi-mita, Tama-ku, Kawasaki 214-8571,
Japan} \email{goto@math.meiji.ac.jp}
\author{J. Hong}
\address{Department of Mathematics, Southern Connecticut State
University, 501 Crescent Street, New Haven, CT 06515-1533, U. S. A.}
\email{hongj2@southernct.edu}
\author{K. Ozeki} \address{Department
of Mathematics, School of Science and Technology, Meiji University,
1-1-1 Higashi-mita, Tama-ku, Kawasaki 214-8571, Japan}
\email{kozeki@math.meiji.ac.jp}
 \author{T.T. Phuong}
\address{Department of Information Technology and Applied
Mathematics, Ton Duc Thang University, 98 Ngo Tat To Street, Ward 19,
Binh Thanh District, Ho Chi Minh City, Vietnam}
\email{sugarphuong@gmail.com}
\author{W. V. Vasconcelos}
\address{Department of Mathematics, Rutgers University, 110
Frelinghuysen Rd, Piscataway, NJ 08854-8019, U. S. A.}
\email{vasconce@math.rutgers.edu}

\thanks{{AMS 2010 {\em Mathematics Subject Classification:} 13H10,
13H15, 13A30.}\\
\\The first author is partially supported by a grant
from the City University of New York PSC-CUNY Research Award
Program-44. The second author is partially supported by Grant-in-Aid
for Scientific Researches (C) in Japan (19540054).  The fourth author
is supported by a grant from MIMS (Meiji Institute for Advanced Study
of Mathematical Sciences).  The fifth author is supported by JSPS
Ronpaku (Dissertation of PhD) Program.  The last author is partially
supported by the NSF}

\begin{abstract}
 The set of the first Hilbert coefficients of parameter ideals relative to a
 module--its Chern coefficients--over a local Noetherian ring codes for considerable
 information about its structure--noteworthy
properties such as that of Cohen-Macaulayness, Buchsbaumness, and of having finitely
 generated local cohomology.
The authors have previously studied the ring case.
 By
developing a robust setting to treat these coefficients for
unmixed rings and modules, the case of modules is analyzed in a more transparent
manner. Another series of integers arise from partial Euler
characteristics and are shown to carry similar properties of the
module.
The technology of homological degree theory is also introduced in order to
derive bounds for these two sets of numbers.

\end{abstract}

\maketitle

 {\em \small Dedicated to Professors N. V. Trung and  G. Valla for their groundbreaking contributions to the theory of Hilbert functions.}

\section{Introduction}

\noindent Let $\RR$  be a Noetherian local ring with maximal ideal $\m$  and let $I$ be an $\mathfrak{m}$-primary ideal. There is a great deal of interest on the set of $I$-good filtrations of $\RR$. More concretely,  on the set of multiplicative, decreasing  filtrations
\[\mathcal{A}=\{ I_n \; \mid \;  I_0=\RR, \; I_{n+1}=I I_n, n\gg 0 \} \]
of $\RR$ ideals which are integral over the $I$-adic filtration, conveniently coded in the corresponding Rees algebra and its associated graded ring
\[ \Rees(\mathcal{A}) = \sum_{n\geq 0} I_nt^n, \quad \gr_{\mathcal{A}}(\RR) = \sum_{n\geq 0} I_n/I_{n+1}. \]

Our focus here is on a set of filtrations both broader and more narrowly
defined.   Let
$M$ be a finitely generated $\RR$-module.
 The Hilbert polynomial of the associated graded module \[
\gr_I(M)= \bigoplus_{n\geq 0} I^nM/I^{n+1}M,\] more precisely the
values of the length $\lambda(M/I^{n+1}M)$ of $M/I^{n+1}M$ for large $n$, can be assembled as \[
P_{M}(n)=\sum_{i=0}^{r} (-1)^i e_i(I,M) {{n+r-i}\choose{r-i}}, \] where
$r=\dim_{\RR} M > 0$.  In most of our discussion,  either $I$ or $M$ is fixed,
and by simplicity we set $e_i(I,M)=e_i(M)$ or $e_i(I,M)=e_i(I)$,
accordingly.  Occasionally the first Hilbert coefficient $e_1(I, M)$
is referred to as the {\em Chern coefficient} of $I$ relative to $M$
(\cite{chern1}).

\medskip

The authors have examined  (\cite{chern3}, \cite{chern2}, \cite{GO}, \cite{chern1})
how the values of $e_1(Q,\RR)$
 codes for  structural information about the ring
$\RR$ itself. More explicitly one defines the set
 \[\Lambda (M) = \{e_1(Q,M) \; \mid \; Q ~\operatorname{is ~a ~parameter
~ideal ~for} M \}
\] and examines what its structure expresses about $M$.
In case $M=\RR$, this set was analyzed for the following extremal  properties:
\begin{enumerate}
\item[{\rm (a)}] $0\in \Lambda(\RR)$;
\item[{\rm (b)}] $\Lambda(\RR)$ contains a single element;
\item[{\rm (c)}] $\Lambda(\RR)$ is bounded.
\end{enumerate}

\medskip

The task of determining the elements of $\Lambda(M)$ has turned out
to be rather daunting.
 More amenable has been the approach to obtain
specialized bounds using cohomological techniques.
An unresolved issue has been to describe the character of the set
$\Lambda(M)$, in particular the role of its extrema and the gap
structure of the set itself.

\medskip

The other invariant of the module $M$ in our investigation is the
following.
Let $Q = (x_1, x_2, \ldots, x_r)$ be a parameter ideal for $M$. We
denote by $\rmH_i(Q;M)$~($i \in \mathbb Z)$ the $i$--th homology module of
the Koszul complex $\rmK_{\bullet}(Q; M)$  generated
by the system $\mathbf{x} = \{x_1, x_2, \ldots, x_r\}$ of parameters of $M$. We put
\[\chi_1(Q;M) =  \sum_{i \ge 1}(-1)^{i-1}\lambda_\RR(\rmH_i(Q;M))\] and call it
the first {\em Euler characteristic} of $M$ relative to $Q$; hence
\[\chi_1(Q;M) = \lambda_\RR(M/QM) -
e_0(Q,M)\] by a classical result of Serre (see \cite{AB},
\cite{Serrebook}).

In analogy to $\Lambda (M)$,  one defines the set
 \[\Xi (M) = \{\chi_1(Q;M) \; \mid \; Q ~\operatorname{is ~a ~parameter
~ideal ~for} M \}
\] and examines again what its structure expresses about $M$.
Most of the properties of this set can be assembled from a diverse
literature, particularly from \cite[Appendice II]{Serrebook}.
 The outcome is a
listing that mirrors, step-by-step, all the properties of the set
$\Lambda(M)$ that we study.

\medskip

We shall now describe more precisely our results. Section 2 starts with  a review of some
elementary computation rules for $e_1(Q, M)$ under hyperplane sections,
more properly modulo superficial elements.
Since part of our goal is to extend to modules our previous results on rings,
 given the ubiquity of the unmixedness hypothesis,  we develop a fresh setting
to treat the module case.
 It made for more transparent proofs. These are carried out in
 Sections 3--5.

\medskip

Section 6 introduces homological degree techniques to obtain
special bounds for the set $\Lambda(M)$. The treatment here is more
general and sharper than in \cite{chern1}.
 Thus in
Corollary~\ref{lambdabarq} it is proved that the set
\[\Lambda_Q (M) = \{e_1(\q,M) : \q ~\operatorname{is ~a ~parameter
~ideal ~for} M ~\operatorname{with ~the ~same ~integral ~closure ~as ~that ~of} Q \},
\]is finite.
 In Section 7, we treat the sets $\Xi(M)$ and $\Xi_Q(M)$, focusing on the
properties that have analogs in $\Lambda(M)$ (see Table 1). In particular, we prove that Euler characteristics can be uniformly bounded by homological degrees (Theorem~\ref{chi1hdeg}). We also consider the numerical function, which we call the {\em Hilbert characteristic} of $M$ with respect to $Q=(\xx)$:
\[ \hh(\xx;M)= \sum_{i=0}^r (-1)^i e_i(Q, M).\]
If the system $\mathbf{x} = \{x_1, x_2, \ldots, x_r\}$ of parameters of $M$ forms a $d$--sequence for $M$,  $\hh(\xx;M)$ has some properties of a
homological degree. They are enough to bound the Betti numbers
$\beta_i(M)$ in terms of $\beta_i(\RR/\m)$, a well-known property of
cohomological degrees. Finally,
 in Section 8, we recast in the context of Buchsbaum-Rim
coefficients several questions treated in this paper.

\medskip

A street view of our results for the convenience of the reader is given in the following table.
 Let $\RR$ be a Noetherian
local ring with infinite residue class field and $M$ a finitely
generated $\RR$-module with $r = \dim_{\RR}M \ge 2$. Let $\mathcal{P}(M)$ be the
collection of systems $\xx=\{x_1, x_2, \ldots, x_r\}$ of parameters of $M$. In
\cite{chern3} and in this paper, the authors study multiplicity
derived numerical
functions
\[ \ff: \mathcal{P}(M) \lar \Bbb N \]
on emphasis on the nature of its range \[\XX_\ff(M)= \{\ff(\xx)\mid
\xx\in \mathcal{P}(M)\}.\]
For the two functions $e_1(\xx,M)$ and $\chi_1(\xx;M)$, more
properly $\ff_1(\xx)=-e_1(\xx,M) $,
and $\ff_2(\xx)=\chi_1(\xx;M)$, respectively,  identical assertions
about the character of $\XX_\ff(M)$ are expressed in the following grid:

\begin{center}
{\small
\begin{tabular}{|c|c|c|c|} \hline
\rule{0pt}{3.2ex}  $\XX_\ff(M)\subseteq [0,\infty )$  & $M$ &  \cite{MSV10} &
  \cite[Appendice II]{Serrebook}     \\ \hline
\rule{0pt}{3.2ex} $0\in \XX_\ff(M)$   & $M$ Cohen-Macaulay & Theorem~\ref{e1zeromod}$^*$    &
\cite[Appendice II]{Serrebook}
     \\ \hline
\rule{0pt}{3.2ex}   $|\XX_\ff(M)| < \infty$  &  $M$ generalized
Cohen-Macaulay   & Theorem~\ref{genCMchern3}$^*$  &
\cite{CST}       \\ \hline
\rule{0pt}{3.2ex}   $|\XX_\ff(M)|=1$ & $M$ Buchsbaum & 
Theorem~\ref{4.1}$^*$    &
\cite{SV2}    \\ \hline
\rule{0pt}{3.2ex}   $|\{\ff(\xx)\mid
\bar{Q}=\bar{(\xx)}\}| < \infty$  & $Q$  &
    Corollary~\ref{lambdabarq}      & Corollary~\ref{chi1intclos}  \\ \hline
\end{tabular}
}
\end{center}

\medskip

{\small TABLE 1. Properties of a finitely generated  module $M$ carried by the values of
either function. An adorned reference [XY]$^*$ requires that the
module $M$ be unmixed. The third and fourth columns  refer to the functions
$\ff_1(\xx)$ and $\ff_2(\xx)$ respectively.
}

\section{Preliminaries}

Throughout this section let $\RR$  be a Noetherian
local ring with maximal ideal $\m$ and let $M$ be a finitely generated $\RR$--module. For basic terminology and  properties of Noetherian rings and
Cohen--Macaulay rings and modules  we make use of \cite{BH} and
\cite{Mat0}.
For convenience of exposition we treat briefly the role of hyperplane
sections in Hilbert functions and examine unmixed  modules. We add
further clarifications when we define homological degrees.

\subsubsection*{Hyperplane sections and Hilbert polynomials.}

\medskip

We need  rules to compute  these coefficients. Typically they involve
so called {\em superficial elements} or {\em filter regular
elements}. We keep
the terminology of {\em generic hyperplane section},
even when  dealing with Samuel's multiplicity with respect to an $\m$--primary ideal
$I$ and its Hilbert coefficients $e_i(M) = e_i(I, M)$. Hopefully this usage will not lead
to undue
confusion.  We say that $h\in I$ is a {\em parameter} for $M$, if $\dim_{\RR}
M/hM< \dim_{\RR} M$.

\medskip

Let us begin with the following.

\begin{Lemma}\label{genhs2}
Let $(\RR,\m)$ be a Noetherian local ring, $I$ an $\m$--primary ideal of $\RR$,
and $M$ a finitely generated
$\RR$--module. Let $h\in I$ and suppose that $\lambda (0:_Mh) < \infty$.
Then we have the following.
\begin{enumerate}
\item[{\rm (a)}] $h$ is a parameter for $M$, if $\dim_{\RR} M > 0$.

\item[{\rm (b)}]
 $\lambda(0:_Mh)\leq \lambda(\H_{\m}^0(M/hM))$.

\item[{\rm (c)}]{\rm \cite[(1.5)]{RV}} If  $\dim_{\RR} M>1$ and $M/hM$ is Cohen--Macaulay, then $M$ is Cohen--Macaulay.
\end{enumerate}
\end{Lemma}

\begin{proof} Suppose that $\dim_{\RR} M > 0$ and let $\p \in \Supp_{\RR}M$ with $\dim \RR/\p = \dim_{\RR} M$. Then $$(0):_{M_\p}\frac{h}{1} = (0),$$ since $\p \ne \m$. As $\dim_{{\RR}_\p} M_\p = 0$, we get $h \not\in \p$. Hence $h$ is a parameter for $M$, if $\dim_{\RR} M > 0$.

We look at the exact sequence
$$
0 \to (0):_M h \to \rmH_\m^0(M) \overset{h}{\to} \rmH_\m^0(M)  \overset{\varphi}{\to}  \rmH_\m^0(M/hM)\to \rmH_\m^1(M) \overset{h}{\to} \rmH_\m^1(M) \to \rmH_\m^1(M/hM) \to \cdots
$$
of local cohomology modules derived from the exact sequence
$$
0 \to (0):_Mh \to M \overset{h}{\to} M \to M/hM \to 0
$$
of $\RR$--modules. We then have $$\lambda ((0):_Mh) = \lambda (\operatorname{Im} \varphi) = \lambda (\rmH_\m^0(M/hM)) - \lambda ((0):_{\rmH_\m^1(M)}h) \le \lambda (\rmH_\m^0(M/hM)).$$
Therefore, if $\dim_{\RR} M > 1$ and $M/hM$ is Cohen--Macaulay, then $h$ is $M$--regular and hence $M$ is Cohen--Macaulay as well.
 \end{proof}

 We will make repeated use of \cite[(22.6)]{Nagata} and   \cite[Section 3]{MSV10}. See also \cite{RV2} and \cite{RV} for a more general version of these results.

\begin{Proposition} \label{genhs} Let $(\RR,\m)$ be a Noetherian local ring, $I$ an $\m$--primary ideal of $\RR$, and $M$ a finitely generated
$\RR$--module with $r = \dim_{\RR} M > 0$.
Let $h \in I$ and assume that $h$ is superficial for $M$ with respect to $I$ $($in
particular
$h\in I\setminus \m I)$.
\begin{enumerate}
\item[{\rm (a)}]  The Hilbert coefficients of $M$ and $M/hM$ satisfy
\begin{eqnarray*} e_i(M)& =& e_i(M/hM) \ \ \text{for} \ \ 0 \le i< r-1 \ \ \text{and} \ \ \\
e_{r-1}(M) &=& e_{r-1}(M/hM) + (-1)^{r} \lambda(0:_Mh).
\end{eqnarray*}

\item[{\rm (b)}]
Let
$ 0 \rar A \rar B \rar C \rar 0$
be an exact sequence of finitely generated $\RR$--modules.
If $t=\dim_{\RR} A< s= \dim_{\RR} B$, then $e_i(B)=e_i(C)$ for $0 \le i< s-t$. In
particular, if $t=0$ and $s \geq 2$, then $e_1(B)=e_1(C)$.
\medskip

\item[{\rm (c)}]\label{e1dim1} If $M$ is a module of dimension $1$ and $I$ is a parameter ideal for $M$, then
\[ e_1(M)=-\lambda(\H_{\m}^0(M)). \]

\item[{\rm (d)}]\label{e1dim2} If $M$ is a module of dimension $2$ and $I$ is a parameter ideal for $M$, then
\[ e_1(M) = e_1(M/hM) + \lambda(0:_Mh) =  - \lambda(\H_{\m}^0(M/hM))+ \lambda(0:_Mh) = -\lambda((0):_{\rmH_\m^1(M)}h).  \]

\end{enumerate}
\end{Proposition}

\begin{proof}
See Proof of Lemma \ref{genhs2} for assertion (d).
\end{proof}

The following Corollary was previously observed in \cite{MSV10}. By induction on $r=\dim_{\RR} M$, it also can be achieved independently, using Proposition~\ref{genhs}.

\begin{corollary}\label{e1par}
If $M$ is a module of positive dimension and $I$ is a parameter ideal for $M$, then $e_1(I, M)\leq 0$.
\end{corollary}

\medskip

\subsubsection*{Unmixed modules.}
We recall the notion of unmixed local rings and modules and develop a setting to
study their Hilbert coefficients.

\begin{Definition}{\rm Let $(\RR, \m)$ be a Noetherian local ring of
dimension $d$. Then we say that $\RR$ is {\em unmixed}, if $\dim \widehat{\RR}/\p=d$ for every $\p\in
\Ass \widehat{\RR}$, where $\widehat{\RR}$
is the
$\m$--adic completion of $\RR$. Similarly, let $M$ be a finitely  generated
$\RR$--module of dimension $r$.  Then we say that
$M$ is {\em unmixed}, if $\dim \widehat{\RR}/\p=r$ for every $\p\in
\Ass_{\widehat{\RR}}\widehat{M}$, where $\widehat{M}$ denotes the $\m$--adic completion of $M$.
}\end{Definition}

Our formulation of unmixedness is the following.

\begin{Theorem} \label{unmixedrep} Let $\RR$ be a Noetherian local ring and $M$ a finitely
generated $\RR$--module with $\dim_{\RR} M = \dim \RR$. Then the following conditions are equivalent{\rm\,:}
\begin{enumerate}
\item[{\rm (i)}] $M$ is an unmixed $\RR$--module{\rm \,;}
\item[{\rm (ii)}] There exists a surjective homomorphism $\SS \to \widehat{\RR}$ of rings together with an embedding $\widehat{M} \hookrightarrow \SS^n$ as an $\SS$--module for some $n > 0$, where $\SS$ is a Gorenstein local ring with $\dim \SS = \dim \RR$.
\end{enumerate}
\end{Theorem}

\begin{proof}
We have only to prove (i) $\Rightarrow$ (ii). We may assume $\RR$ is complete. Thanks to Cohen's structure theorem of complete local rings, we can choose a surjective homomorphism $\SS \to \RR$ of rings such that $\SS$ is a Gorenstein local ring with $\dim \SS = \dim \RR$. Then, because $\Ass_{\SS}M \subseteq \Ass \SS$ and the $\SS_\p$--module $M_\p$ is reflexive for all $\p \in \Ass_{\SS}M$, the canonical map
$$
M  \to \Hom_{\SS}(\Hom_{\SS}(M, \SS), \SS)
$$
is injective, while we get an embedding
$$
\Hom_{\SS}(\Hom_{\SS}(M, \SS), \SS) \hookrightarrow \SS^n
$$
for some $n > 0$, because $\Hom_{\SS}(M, \SS)$ is a finitely generated $\SS$--module. Hence the result.
\end{proof}

\medskip

\begin{Corollary}[\cite{GNa}]\label{fg} Let $(\RR, \m)$ be a Noetherian local ring and $M$ a finitely
generated $\RR$--module with $\dim_{\RR} M = \dim \RR \geq 2$. If $M$ is an unmixed
 $\RR$--module, then $\H_{\m}^1(M)$ is finitely generated.
\end{Corollary}

\begin{proof}
We may assume $\RR$ is complete. We maintain the notation in Proof of Theorem \ref{unmixedrep} and let $\n$ denote the maximal ideal of $\SS$. Then, applying the functors $\rmH_\n^i(*)$ to the exact sequence
$$
0 \to M \to \SS^n \to C \to 0
$$
of $\SS$--modules, we get $\rmH_\m^1(M) \cong \rmH_\n^0(C)$, because $\depth \SS \ge 2$. Hence $\rmH_\m^1(M)$ is finitely generated.
\end{proof}

\medskip

\section{Vanishing of $e_1(Q,M)$}

Let $\RR$ be a Noetherian local ring with maximal ideal $\m$ and $M$ a finitely
generated $\RR$--module with $r= \dim_{\RR} M$. Recall that a parameter ideal for $M$ is an ideal $Q=(x_1, x_2, \ldots, x_r) \subseteq \m$ in $\RR$ with $\lambda(M/QM)< \infty$.

\medskip

\begin{Theorem}\label{e1zeromod}
Let $\RR$ be a Noetherian local ring  and $M$
a finitely generated $\RR$--module with $\dim_{\RR} M \geq 2$. Suppose that $M$ is unmixed and let $Q$ be a parameter ideal for $M$. Then the following conditions are equivalent{\rm \,:}
\begin{enumerate}
\item [{\rm (i)}] $M$ is a Cohen--Macaulay $\RR$--module{\rm \,;}
\item[{\rm (ii)}] $e_1(Q, M)=0$.
\end{enumerate}
\end{Theorem}

\begin{proof} We set $e_1(Q)=e_1(Q,M)$.
  It is enough to
show that if $M$ is not
Cohen--Macaulay, then $e_1(Q) <0$.  We may assume that $\RR$ is
complete with an infinite residue
field and $\dim \RR=\dim M $.

\medskip

Choose a Gorenstein  local ring $(\SS,\n)$ and a surjection $\SS\rar \RR$,
with $\dim \SS=\dim \RR$.
If $Q$ is a parameter ideal of $\RR$, there exists a parameter
ideal $\q$ of $\SS$ such that $\q \RR =Q$ (\cite[Lemma 3.1]{chern2}).
Therefore the associated graded module  of $Q$  relative to $M$ is isomorphic to the
associated graded module of $\q$ with respect to the $\SS$--module $M$:
\[ \gr_Q(M) \simeq \gr_{\q}(M), \] which implies that \[ e_1(Q) =
e_1(\q, M), \] where $e_1(\q, M)$ denotes the first Hilbert
coefficient of $\q$ with respect to the $\SS$--module $M$.

\bigskip

Consider the exact sequence of $\SS$--modules obtained from
Proposition~\ref{unmixedrep}: \[ 0 \rar M \rar \SS^n
\rar C \rar 0.  \] Let $y$ be a superficial element for $\q$ with
respect to $M$ such that $y$ is part of a minimal generating set of
$\q$. We may assume that $y$ is a nonzero divisor on $M$.
By tensoring
the exact sequence of $\SS$--modules with $\SS/(y)$, we get \[ 0 \rar T
=\Tor_1^{\SS}(\SS/(y), C) \rar M/yM \stackrel{\zeta}{\rar}
\SS^n/y \SS^n \rar
C/yC \rar 0.  \] Let $M\,'= M/yM$ and $N=\im(\zeta)$ and consider
the short exact sequence: \[ 0 \rar T \rar M\,' \rar N \rar 0. \]
Then either $T=0$ or $T$ has finite length $\l(T) < \infty$. Note
that $N$ is an unmixed $\SS/(y)$-module.

\bigskip

\noindent We use induction on $d=\dim M$ to show that if $M$ is
not Cohen--Macaulay, then $e_1(\q, M) <0$.

\bigskip

\noindent Let $d=2$ and $\q=(y, z)$. Then $T \neq 0$ so that $\l(T) < \infty$.
Applying the Snake Lemma to \[
\begin{CD}
 0 @>>> T \cap z^n M\,' @>>> z^n M\,' @>>> z^n N @>>> 0 \\
  & & @VVV  @VVV  @VVV \\
0 @>>> T @>>> M\,' @>>> N @>>> 0
\end{CD}
\] we get, for sufficiently large $n$,
\[  \l(M\,'/z^n M\,') = \l(T) + \l(N/z^n N).  \]
Computing the Hilbert polynomials, we have
\[ e_1(\q, M) = e_1(\q/(y), M/yM) = -\l(T) <0. \]

\bigskip

\noindent Now suppose that $d \geq 3$. From the exact sequence \[ 0 \rar T \rar
M'=M/yM \rar N \rar 0, \] we have \[ e_1(\q, M) = e_1(\q/(y) ,
M/yM ) = e_1(\q/(y) , N).  \] By an induction argument, it is
enough to show that $N$ is not Cohen--Macaulay since
$\dim({\SS}/(y))=d-1$.

\bigskip

\noindent Suppose that $N$ is Cohen--Macaulay. Let $\n$ be the maximal ideal
of $\SS/y\SS$. From the exact sequence \[ 0 \rar T \rar M'=M/yM \rar
N \rar 0, \] we obtain the long exact sequence: \[ 0 \rar
\H_{\n}^0(T) \rar \H_{\n}^0(M\,') \rar \H_{\n}^0(N) \rar \H_{\n}^1(T)
\rar \H_{\n}^1(M\,') \rar \H_{\n}^1(N).  \]By the assumption that
$N$ is Cohen--Macaulay of dimension $d-1 \geq 2$ and the fact that
$T$ is a torsion module, we get

\[
0 \rar T \simeq \H_{\n}^0(M\,') \rar 0 \rar 0 \rar \H_{\n}^1(M\,') \rar 0.
\]

\noindent From the
exact sequence \[0 \rar M \stackrel{\cdot y}{\rar} M \rar M'= M/yM \rar
0, \] we obtain the following exact sequence: \[ 0 \rar {T \simeq
\H_{\n}^0(M\,')} \rar \H_{\n}^1(M) \stackrel{\cdot y}{\rar} \H_{\n}^1(M)
\rar {\H_{\n}^1(M\,')=0}.  \] Since $\H_{\n}^1(M) $ is finitely
generated by Corollary \ref{fg} and ${\ds \H_{\n}^1(M)=y\H_{\n}^1(M)} $, we have ${\ds
\H_{\n}^1(M)=0}$. This means that $T=0$. Therefore \[ {\ds 0 \rar T=0
\rar M/yM \simeq N \rar 0.} \] Since $N$ is
Cohen--Macaulay, $M/yM$ is Cohen--Macaulay.  Since $y$ is
regular on $M$, $M$ is Cohen--Macaulay, which is a contradiction. \end{proof}

\bigskip

\begin{Example} {\rm (\cite{chern1}). Let  $M=\RR=k[[x,y,z]]/(z(x,y,z))$.
Then $H^0_{\m}(\RR)=(z)$, and $\SS=\RR/H^0_{\m}(\RR)\simeq k[[x,y]]$ is Cohen-Macaulay.
If $Q$ is a parameter ideal of $\RR$, then $e_1(Q,\RR)=e_1(Q\SS,\SS)=0$.
Hence $e_1(Q,\RR)=0$, but $\RR$ is not Cohen-Macaulay. Therefore the unmixdness condition is necessary Theorem \ref{e1zeromod}. }
\end{Example}

\bigskip

Let us list some consequences of Theorem \ref{e1zeromod}.
Let $\RR$ be a Noetherian local ring and $M$ a finitely generated
$\RR$-module. We put
\[ \Assh_{\RR} M = \{ \p \in \Ass_{\RR} M \; \mid \; \dim \RR/\p=\dim_{\RR} M\}.\]
Let ${ (0_M) = \bigcap_{\p \in \Ass_{\RR} M} M(\p) }$ be a primary
decomposition of $0_M$ in $M$, where $M(\p)$ is a $\p$--primary submodule of $M$ for each $\p \in \Ass_{\RR}M$.  We put \[ \mathrm{U}_M(0) = \bigcap_{\p \in
\Assh_{\RR} M} M(\p) \]
and call it the {\em unmixed component} of $M$. We then have the following.

\medskip

\begin{Lemma}\label{comparison}
Let $\RR$ be a Noetherian local ring and $M$  a finitely
generated $\RR$--module with $r = \dim_{\RR}M > 0$. Let $Q$ be a parameter
ideal for $M$. Let $U=\mathrm{U}_M(0)$ and suppose that $U \neq (0)$.  We put $N = M/U$. Then
the following assertions hold.
\begin{enumerate}
\item[{\rm (a)}] $\dim_{\RR} U < \dim_{\RR} M.$
\item[{\rm (b)}]
\[
e_1(Q,M) = \left\{\begin{array}{ll}
e_1(Q,N) &\quad \mbox{\rm if} \;\; \dim_{\RR} U \leq r-2, \\& \\
e_1(Q,N)-s_0 &\quad \mbox{\rm if} \;\; \dim_{\RR} U=r-1,
\end{array}
\right.
\]where $s_0 \geq 1$ denotes the multiplicity of  the graded $\gr_Q(\RR)$--module ${ \bigoplus_{n \ge 0} U/(Q^{n+1}M\cap U)}.$

\item[{\rm (c)}] $e_1(Q,M) \leq e_1(Q,N)$ and the equality $e_1(Q,M) = e_1(Q,N)$ holds if and only if $\dim_{\RR} U \leq r-2$.
\end{enumerate}
\end{Lemma}

\medskip

\begin{proof} (a) This is clear, since $\mathrm{U}_\p = (0)$ for all $\p \in \Assh_{\RR}M$.

\medskip

\noindent (b) We write
\[\l_{\RR}(U/(Q^{n+1}M\cap U))
 = s_0\binom{n + t }{t} -s_1\binom{n + t-1}{t-1} + \cdots +
 (-1)^ts_t\]
for $n \gg 0$ with integers $\{s_i\}_{0 \le i \le t}$, where $t=\dim_{\RR}
U$.  Then the claim follows from the exact sequence $0 \rar U \rar M
\rar N \rar 0$ of $\RR$--modules, which gives \[ \l_{\RR}(M/Q^{n+1}M) =
\l_{\RR}(N/Q^{n+1}N) + \l_{\RR}(U/(Q^{n+1}M\cap U)), \;\; \forall \; n \ge 0. \]

\medskip

\noindent (c) This follows from (b) and the fact that $s_0 \ge 1$.
\end{proof}

\medskip

\begin{Theorem}\label{1.2-2}
Let $\RR$ be a Noetherian local ring and $M$ a finitely
generated $\RR$--module with $r=\dim_{\RR}M \geq 2$. Suppose that $\RR$ is a
homomorphic image of a Cohen--Macaulay ring. Let $U=\mathrm{U}_M(0)$ and let
$Q$ be a parameter ideal for $M$.  Then the following conditions are
equivalent{\rm \,:}
\begin{enumerate}
\item[{\rm (i)}] $e_1(Q,M)=0${\rm \,;}
\item[{\rm (ii)}] $M/U$ is a Cohen--Macaulay $\RR$--module and $\dim_{\RR} U \leq r-2$.
\end{enumerate}
\end{Theorem}

\demo It is enough to prove (i) $\Rightarrow$ (ii). If $\dim_{\RR} U = r-1$, then by (i) and Lemma~\ref{comparison}--(b), we obtain
$0 \geq e_1(Q,M/U)=s_0 \geq 1$, which is a contradiction. Hence $\dim_{\RR} U \leq r-2$. This means that ${\ds 
0= e_1(Q,M) = e_1(Q,M/U)}$. By Theorem~\ref{e1zeromod}, $M/U$ is Cohen--Macaulay.
\QED

\medskip

The implication (i) $\Rightarrow$ (ii) in Theorem
\ref{1.2-2} is not true in general without the assumption that $\RR$ is
a homomorphic image of a Cohen--Macaulay ring. See \cite[Remark
2]{chern3} for an example. The following corollary gives a characterization of Cohen--Macaulayness.

\begin{Corollary}\label{char-CM}
Let $\RR$ be a Noetherian local ring, $M$ a finitely generated
$\RR$--module with $r=\dim_{\RR}M >0$, and $Q$ a parameter ideal for
$M$. Let $1 \leq k \leq r$ be an integer and assume that $e_i(Q,M)=0$
for all $1 \leq i \leq k$. Then $$\dim_{\widehat{\RR}}
\mathrm{U}_{\widehat{M}} (0) \leq r - (k+1)\ \ \text{and}\ \ \ \H_{\fkm}^{r-j}(M) = (0)$$ for
all $1 \leq j \leq k$. In particular, if $k = r$, then $M$ is a
Cohen--Macaulay $\RR$--module.
\end{Corollary}

\begin{proof}
We may assume that $\RR$ is complete. We put $U =\mathrm{U}_M (0)$ and $N = M/U$.
Then $e_0(Q,M)=e_0(Q,N)$, since $\dim_{\RR} U < r$. Therefore, by Theorem~\ref{1.2-2} $N$ is a Cohen--Macaulay $\RR$--module, so that we have exact sequences
\[0 \to U/Q^{n+1}U \to
M/Q^{n+1}M \to N/Q^{n+1} N \to 0\]
 of $\RR$--modules for all $n \ge 0$. Hence, computing
Hilbert polynomials, we get $\dim_{\RR} U \leq r -(k+1)$. Let $1 \leq
j \leq k$. Then $\H_{\fkm}^{r-j}(U)= (0)$, since $\dim_{\RR} U <r-j$, while
$\H_{\fkm}^{r-j}(N)=(0)$, as $N$ is a Cohen--Macaulay $\RR$--module with $\dim_{\RR}N = r$. Thus $\H_{\fkm}^{r-j}(M) = (0)$ as claimed.
\end{proof}

\medskip

Let $\RR$ be a Noetherian local ring and $M$ a finitely generated $\RR$--module.  In \cite{chern3} the authors examined  the rings with $e_1(Q,R)$  vanishing. Here we briefly extend this theory to modules.
Let us begin with the definition.

\begin{Definition}
A finitely generated $\RR$--module $M$ is called a {\em Vasconcelos}
module\footnote{The terminology is due to the other five authors.},
if either $\dim_{\RR} M = 0$, or $\dim_{\RR} M > 0$ and $e_1(Q,M) = 0$ for
some parameter ideal $Q$ for $M$.
\end{Definition}

\noindent
Every Cohen--Macaulay module is by definition  Vasconcelos.
Here is a basic characterization. We omit the proof, since it is similar to those in the ring case.

\begin{Theorem}\label{vchar} Let
$(\RR, \m)$ be a Noetherian local ring and $M$ a finitely
generated $\RR$--module with $r=\dim_{\RR} M\geq 2$. Let $U = \mathrm{U}_{\widehat{M}}(0)$ be the unmixed component of $(0)$ in the $\m$--adic completion $\widehat{M}$ of $M$.
Then the  following conditions are
equivalent{\rm \,:}
\begin{enumerate}
\item[{\rm (i)}] $M$ is a Vasconcelos $\RR$--module{\rm \,;}
\item[{\rm (ii)}] $e_1(Q,M) = 0$ for every parameter ideal $Q$ for $M${\rm \,;}
\item[{\rm (iii)}] $\widehat{M}/\mathrm{U}_{\widehat{M}}(0)$ is a Cohen--Macaulay
$\widehat{\RR}$--module and $\operatorname{dim}_{\widehat{\RR}}
\mathrm{U}_{\widehat{M}}(0) \le r-2${\rm \,;}
\item[{\rm (iv)}] There exists a proper $\widehat{\RR}$--submodule $L$ of $\widehat{M}$ such that
$\widehat{M}/L$ is a Cohen--Macaulay $\widehat{\RR}$--module with
$\operatorname{dim}_{\widehat{\RR}}L \le r-2$.
\end{enumerate}
When this is the case, $\widehat{M}$ is  a Vasconcelos
$\widehat{\RR}$--module and $\H_{\m}^{r-1}(M) = (0)$.
\end{Theorem}

\medskip

\begin{Remark}{\rm
Several properties of Vasconcelos rings such as \cite[3.5, 3.8, 3.9, 3.10, 3.11, 3.12, 3.13, 3.15, 3.16, 3.17]{chern3} can be all extended to Vasconcelos modules.
}\end{Remark}

\section{Generalized Cohen--Macaulayness of modules with $\Lambda
(M)$ finite}

Let $\RR$ be a Noetherian local ring with maximal ideal $\m$ and $M$ a finitely generated $\RR$--module with $r =
\dim_{\RR} M > 0$.  In this section we study the problem of when the set
\[\Lambda (M) = \{e_1(Q,M) \mid Q ~\operatorname{is ~a ~parameter
~ideal ~for} M \}\] is finite. Part of the motivation comes from the
fact that generalized Cohen--Macaulay modules have this property.
Recall that $M$ is said to be generalized Cohen--Macaulay, if all the local
cohomology modules $\{\H_\m^i(M) \}_{0 \le i < r}$ are finitely
generated (see \cite{CST} where these modules originated).

\medskip

Assume that $M$ is a generalized Cohen--Macaulay $\RR$--module with $r
=\dim_{\RR}M \ge 2$ and put \[s=\sum_{i=1}^{r-1}\binom{r-2}{i-1}h^i(M),\] where
$h^i(M)=\l_\RR(\H_\fkm^i(M))$ for each $i \in \mathbb{Z}$.  If $Q$ is a
parameter ideal for $M$, by the proof of \cite[Lemma 2.4]{GNi} we have
that $ e_1(Q,M) \ge -s.$ Since $e_1(Q,M)\leq 0$ by Corollary
\ref{e1par}, it follows that $\Lambda (M)$ is a finite set.

\medskip

Let us establish here that if $M$ is unmixed and $\Lambda (M)$ is
finite, then $M$ is indeed a generalized Cohen--Macaulay $\RR$--module (Proposition
\ref{flc}).

\medskip

Assume now that $\RR$ is a homomorphic image of a Gorenstein local ring and that $\Ass_{\RR}M = \Assh_{\RR}M$. Then $\RR$ contains a system $x_1, x_2, \ldots, x_r$ of parameters of $M$ which forms a strong $d$-sequence for $M$, that is, the sequence $x_1^{n_1}, x_2^{n_2}, \ldots, x_r^{n_r}$ is a $d$-sequence for $M$ for all integers $n_1, n_2, \ldots, n_r \ge 1$ (see \cite[Theorem 2.6]{Cu} or \cite[Theorem 4.2]{Kw} for the existence of such systems of parameters). For each integer $q \ge 1$
let $\Lambda_q(M)$ be the set of values $e_1(Q,M)$, where $Q$ runs over the parameter ideals for $M$ such that $Q \subseteq \m^q$ and $Q=(x_1, x_2, \ldots, x_r)$ with $x_1, x_2, \ldots, x_r$ a $d$-sequence  for $M$. We then have $\Lambda_q(M) \ne \emptyset$, $\Lambda_{q+1}(M) \subseteq \Lambda_q(M)$ for all $q \ge 1$, and $\alpha \le 0$ for every $\alpha \in \Lambda_q(M)$ (Corollary \ref{e1par}).

\medskip

The following result plays a key role in our argument. The proof which we present here is based on Theorem  \ref{unmixedrep} and slightly different from that of the ring case.

\begin{Lemma}\label{key}
Let $(\RR,\m)$ be a Noetherian local ring and assume that $\RR$ is a homomorphic image
of a Gorenstein ring. Let $M$ be a finitely generated
$\RR$--module with $r = \dim_{\RR} M \geq 2$ and $\operatorname{Ass}_\RR
M= \operatorname{Assh}_\RR M$. Assume that $\Lambda_q(M)$ is
a finite set for some integer $q \ge 1$ and put $\ell =
-\operatorname{min}\Lambda_q(M)$.  Then $\m^{\ell}\H_\m^i(M) = (0)$ for
all $i \ne r$ and hence all the local cohomology modules
$\{\H_\m^i(M)\}_{0 \le i < r}$ are finitely generated.
\end{Lemma}

\begin{proof}
Passing to the ring $\RR/[(0):_{\RR}M]$, we may assume that $\RR$ is a Gorenstein ring with $\dim \RR = \dim_{\RR} M = r$. Enlarging the residue class field $\RR/\m$ of $\RR$  if necessary, we may assume the field $\RR/\m$ is infinite. By Corollary \ref{fg} $\rmH_\m^1(M)$ is finitely generated, since $M$ is unmixed.

Suppose that $r = 2$. We put $\ell' = \l (\rmH_\m^1(M))$. Let $Q = (x, y) \subseteq \m^q$ be a
system of parameters for $M$ such  that $Q\rmH_\m^1(M) = (0)$ and $x,y$ is a $d$-sequence
for $M$. Then $x$ is superficial for $M$ with respect to $Q$. Hence by Proposition \ref{genhs} (d) we get $e_1(Q,M) =
-\l (\rmH_\m^1(M) ) =-\ell'.$ Thus $\ell \ge \ell'$, as $-\ell' =e_1(Q,M) \in \Lambda_q(M)$. Hence $\m^{\ell}\rmH_\m^1(M) = (0)$, because $\m^{\ell'}\rmH_\m^1(M) = (0)$.

Suppose that $r \ge 3$ and that our assertion holds true for
$r-1$. We have an exact sequence
$$(\sharp) \ \ \
0 \to M \to \RR^n \to C \to 0
$$
of $\RR$--modules by Theorem  \ref{unmixedrep}. Choose an $\RR$-regular element $x \in \RR$ so that $x$ is superficial both for $M$ and $C$ with respect to $\m$. Let us fix an integer $m \ge 1$. We put $y = x^m$, $N = M/yM$,  and look at the exact sequence
$$
0 \to (0):_Cy \to N \overset{\varphi}{\to} (\RR/y\RR)^n \to C/yC \to 0
$$
of $\RR$--modules obtained by sequence $(\sharp)$. Let $L = \operatorname{Im} \varphi$. Then $\dim_{\RR}L = r-1$, $\Ass_{\RR}L = \Assh_{\RR}L$, and $\rmH_\m^0(N)\cong (0):_Cy$, because $L$ is an $\RR$--submodule of $(\RR/y\RR)^n$ and $\l ((0):_Cy) < \infty$. Hence $L \cong N/\rmH_\m^0(N)$.

Let $q' \ge q$ be an integer such that ${\m}^{q'}N \cap \H_\m^0(N) = (0)$. Let $y_2, y_3, \ldots, y_r \in \m^{q'}$ be a system of parameters for $L$ and assume that $y_2, y_3, \ldots, y_r$ is a $d$-sequence for $L$. Then, since $(y_2, y_3, \ldots, y_r)N \cap \H_\m^0(N) = (0)$, we get $y_2, y_3, \ldots, y_r$ forms a $d$-sequence for $N$ also. Therefore, since $y$ is $M$-regular, the sequence $y_1=y, y_2, \ldots, y_r$ forms a $d$-sequence for $M$, whence $y_1$ is superficial for $M$ with respect to $Q = (y_1, y_2, \ldots, y_r)$. Consequently
$$e_1((y_2, y_3, \ldots, y_r),L)=e_1((y_2, y_3, \ldots, y_r),N) = e_1(Q,M) \in \Lambda_q(M),$$ so that $\Lambda_{q'}(L) \subseteq  \Lambda_q(M).$
Hence, because  the set  $\Lambda_{q'}(L)$ is finite, the hypothesis of induction on $r$ yields that $\m^{\ell''}\H_\m^i(L) = (0)$ for all $i \ne r-1$, where $\ell'' = -\operatorname{min} \Lambda_{q'}(L)$.
Thus,  because $\ell'' \le \ell$,  $\m^{\ell}\H_\m^i(L) = (0)$ for all $i \ne r-1$. Hence $\m^{\ell}\rmH_\m^i(N) = (0)$ for all $1 \le i < r-1$, because $\rmH_\m^i(N) \cong \rmH_\m^i(L)$ for $i \ge 1$.

 Look now at the exact sequence
\[(\sharp \sharp)
\ \ \cdots \to \H_\m^1(M)
\overset{x^m}{\to} \H_\m^1(M) \to \H_\m^1(N) \to \cdots \to
\H_\m^i(N) \to \H_\m^{i+1}(M) \overset{x^m}{\to} \H_\m^{i+1}(M) \to
\cdots
\]
of local cohomology modules. We then have
\[
\m^{\ell}\left[(0):_{\H_\m^{i+1}(M)}x^m\right] = (0)
\]
for all integers $1 \le i \le r-2$ and $m \ge 1$, since
$\m^{\ell}\H_\m^i(N) = (0)$ for all $1 \le i \le r - 2$.  Thus
$\m^{\ell}\H_\m^{i+1}(M)= (0)$, because
\[
\H_\m^{i+1}(M) = \bigcup_{m \ge 1}\left[(0):_{\H_\m^{i+1}(M)}\m^m\right].
\]
On the other hand, from sequence $(\sharp \sharp)$ we
get the embedding $\H_\m^1(M) \subseteq \H_\m^1(N)$, choosing the integer $m \ge 1$ so that
$x^m\H_\m^1(M) = (0)$. Hence $\m^{\ell}\H_\m^{1}(M) = (0)$, which
completes the proof of Lemma \ref{key}.
\end{proof}

\medskip

Since $\Lambda (M)=\Lambda
(\widehat{M})$, passing to the completion $\widehat{M}$ of $M$ and  applying
Lemma~\ref{key}, we readily get the following.

\begin{Proposition}\label{flc}
Let $(\RR,\m)$ be a Noetherian local ring and $M$ a finitely
generated unmixed $\RR$--module with $r = \dim_{\RR} M \geq 2$. Assume that
$\Lambda (M)$ is a finite set and put $\ell = -\operatorname{min}
\Lambda (M)$. Then $\m^{\ell}\H_\m^i(M) = (0)$ for every $ i \ne r$, so that $M$ is a generalized Cohen--Macaulay $\RR$--module.
\end{Proposition}

We conclude this section with a characterization of $\RR$--modules for
which $\Lambda(M)$ is finite.

Let us note the following with a brief proof.

\begin{Lemma}\label{e1finite}
Let $\RR$ be a Noetherian local ring and  $M$  a finitely
generated $\RR$--module with $r = \dim_{\RR} M \geq 2$. Assume that there
exists an integer $t \geq 0$ such that $e_1(Q,M) \geq -t$ for every
parameter ideal $Q$ for $M$. Then $\dim_{\RR} \mathrm{U}_M(0) \leq r-2$.
\end{Lemma}

\begin{proof}
Let $U=\mathrm{U}_M(0)$ and $N = M/U$. Assume that $\dim_{\RR} U=r-1$. Choose a
system $x_1, x_2, \dots , x_r$ of parameters of $M$ such that $x_r U = (0)$. Let $\ell > t$ be an integer and put
$Q = (x_1^\ell, x_2, \dots , x_r)$. Then we get
exact sequences
\[0 \to U/(Q^{n+1}M \cap U) \to
M/Q^{n+1}M \to N/Q^{n+1}N \to 0\]
 of $\RR$--modules for all $n \geq 0$. Let us take an integer $k \geq 0$ so
that $$Q^nM \cap U = Q^{n-k}(Q^kM \cap U)$$ for $n \geq k$ and consider $U' = Q^k M \cap U$. Let $\q = (x_1^\ell, x_2, \dots ,
x_{r-1})$. We then have $$Q^{n-k}U' = \q^{n-k} U',$$ as $x_r U =(0)$.
Hence for all $n \geq k$
\[\l_\RR (M/Q^{n+1}M) =
\l_\RR(N/Q^{n+1}N) + \l_\RR (U'/\q^{n-k+1}U') + \l_\RR (U/U'),\] which
yields  $-t \leq e_1(Q,M) = e_1(Q,N) - e_0(\q,U').$
Hence $$-t \leq e_1(Q,M) = e_1(Q,N) - e_0(\q,U),$$
because $e_0(\q,U) = e_0(\q,U')$ (remember that $\lambda (U/U') < \infty$). Therefore, since
$e_1(Q,N) \leq 0$ by Corollary~\ref{e1par},
we get
\[\ell \leq
\ell e_0((x_1, x_2, \dots , x_{r-1}),U) = e_0(\q,U) \leq e_1(Q,N) +
t \leq t,\] which is impossible. Thus $\dim_{\RR} U \leq r-2$.
\end{proof}

\medskip

\begin{Remark}\label{e1geqt}
{\rm
Let $\RR$ be a Noetherian local ring and $M$ a finitely generated $\RR$--module with $r = \dim_{\RR} M \geq 2$. Assume that $\dim_{\RR}\mathrm{U}_M(0) \leq r-2$. Let $\q$ be a parameter ideal for $N=M/\mathrm{U}_M(0)$. Then one can find a parameter ideal $Q$ for $M$ with $QN = \q N$, so that $e_1(\q,N) = e_1(\q,M)$ by Lemma $\ref{comparison}$. Hence $\Lambda (M) = \Lambda (N)$.}
\end{Remark}

\medskip

The goal of this section is the following.

\medskip

\begin{Theorem} \label{genCMchern3}
Let $(\RR,\m)$ be a Noetherian local ring and $M$  a finitely
generated $\RR$--module with $r = \dim_{\RR} M \geq 2$. Let
$U=\mathrm{U}_{\widehat{M}}(0)$ denote the unmixed component of $(0)$ in the
$\fkm$-adic completion $\widehat{M}$ of $M$. Then the following conditions are equivalent{\rm \,:}
\begin{enumerate}
\item[{\rm (i)}] $\Lambda(M)$ is a finite set{\rm \,;}
\item[{\rm (ii)}] $\widehat{M}/U$ is a generalized Cohen--Macaulay $\widehat{\RR}$--module and $\dim_{\widehat{\RR}}
U \leq r-2$.
\end{enumerate}
When this is the case, one has the estimation $$0 \ge e_1(Q,M) \geq
-\sum_{i=1}^{r-1} \binom{r-2}{i-1} h^i(\widehat{M}/U)$$ for every
parameter ideal $Q$ for $M$.
\end{Theorem}

\begin{proof}
We may assume that $\RR$ is complete.

\noindent (i) $\Rightarrow$ (ii) Since the set $\Lambda(M)$ is  finite,
by Proposition \ref{e1finite} we get $\dim_{\RR} U \leq
r-2$. By Remark \ref{e1geqt} the set $\Lambda(M/U)$ is
 finite, so that $M/U$ is a generalized
Cohen--Macaulay $\RR$--module by Proposition \ref{flc}.

\medskip

\noindent (ii) $\Rightarrow$ (i) By \cite[Lemma 2.4]{GNi} the set
$\Lambda(M/U)$ is finite and hence the set $\Lambda(M)$ is also finite by Lemma \ref{comparison}.

\medskip

\noindent See \cite[Lemma 2.4]{GNi} for the last assertion.
\end{proof}

\section{Buchsbaumness of modules possessing constant first Hilbert
coefficients of parameters}

Let $\RR$ be a Noetherian local ring with maximal ideal $\m$ and
$M$  a finitely generated $\RR$--module with $r = \dim_{\RR} M > 0$.  In
this section we study the problem of when $e_1(Q,M)$ is independent
of the choice of parameter ideals $Q$ for $M$. Part of the motivation
comes from the fact that Buchsbaum modules have this property. We
establish here that if $e_1(Q,M)$ is constant and $M$ is unmixed,
then $M$ is indeed a Buchsbaum $\RR$--module (Theorem \ref{4.1}). See \cite{GO} for
the ring case.

\medskip

First of all let us  recall some  definitions. A system $x_1,
x_2, \ldots, x_r$ of parameters of $M$ is said to be {\it standard}, if it forms a
$d^+$-sequence for $M$, that is, $x_1, x_2, \ldots, x_r$ forms a strong
$d$-sequence for $M$ in any order. Remember that $M$ possesses a standard
system of parameters if and only if $M$ is a generalized
Cohen--Macaulay $\RR$--module (\cite{T}).

Let $Q$ be a parameter ideal for $M$. Then we say that $Q$ is standard, if it is
generated by a standard system of parameters of $M$. Remember that $Q$ is
standard if and only if the equality
\[\l_\RR(M/QM) - e_0(Q,M) =
\sum_{i=0}^{r-1}\binom{r-1}{i}h^i(M):={\mathbb I}(M)\] holds (\cite[Theorem 2.1]{T}). It is known that every system of parameters of $M$ contained in a standard parameter ideal for $M$ is standard (\cite{T}).

\medskip

Suppose  that $M$ is a generalized Cohen--Macaulay $\RR$--module with $r =\dim_{\RR}M\ge
2$ and $s=\sum_{i=1}^{r-1}\binom{r-2}{i-1}h^i(M)$. If $Q$ is a
parameter ideal for $M$, then by \cite[Lemma 2.4]{GNi} we get $
e_1(Q,M) \ge -s, $ where the equality holds if $Q$ is standard
(\cite[Korollar 3.2]{Sch1}).

\medskip

We say that our $\RR$--module $M$ is Buchsbaum, if every
parameter ideal for $M$ is standard. Hence, if $M$ is a Buchsbaum $\RR$-module with $r=\dim_{\RR}M \ge 2$, then $M$ is a generalized Cohen--Macaulay $\RR$--module with
$$e_1(Q,M)=-\sum_{i=1}^{r-1}\binom{r-2}{i-1}h^i(M)$$
for every parameter ideal $Q$. See
\cite{SV2} for a detailed theory of Buchsbaum rings and modules.

\medskip

We begin with the following two results, whose proofs are similar to those in the ring case (see {\cite[Lemma 4.5]{chern3}} and  {\cite[Proposition 2.3]{GO}}).

\begin{Lemma}\label{4.5}
Let $(\RR,\m)$ be a Noetherian local ring and $M$ a generalized
Cohen--Macaulay $\RR$--module with $r=\dim_{\RR} M \geq 2$ and $\depth_{\RR}M > 0$.
Let Q be a parameter ideal for $M$ such that $e_1(Q,M)
=-\sum_{i=1}^{r-1} \binom{r-2}{i-1}h^i(M)$. Then $Q\H_\fkm^i(M) =
(0)$ for all $1 \leq i \leq r-1$.
\end{Lemma}

\medskip

For each $x\in \m$, we put
$\mathrm{U}_M(x) := \bigcup_{n \geq 0}[xM :_M \fkm^n]$.

\begin{Proposition}\label{4.6}
Let $(\RR,\m)$ be a Noetherian local ring and $M$ a generalized
Cohen--Macaulay $\RR$--module with $r= \dim_{\RR} M \geq 3$ and $\depth_{\RR}M >
0$. Let $Q=(x_1, x_2, \dots , x_r)$ be a parameter ideal for $M$. Assume that $(x_1, x_r) \H_\fkm^1(M) = (0)$ and that the parameter ideal $(x_1, x_2, \dots , x_{r-1})$ for the generalized Cohen--Macaulay $\RR$--module $M/\mathrm{U}_M(x_r)$ is standard. Then
$\mathrm{U}_M(x_1) \cap QM =x_1M.$
\end{Proposition}

\medskip

We then have the following, which is the key in our argument. The proof is similar to the ring case  {\cite[Theorem 2.1]{GO}} but let us note a brief proof in order to see how we use the previous results Lemma \ref{4.5} and Proposition \ref{4.6}.

\medskip

\begin{Theorem}\label{4.7}
Let $(\RR,\m)$ be a Noetherian local ring and let $M$ be a generalized
Cohen--Macaulay $\RR$--module with $r= \dim_{\RR} M \geq 2$ and $\depth_{\RR}M >
0$.  Let $Q$ be a parameter ideal for $M$. Then the following
conditions are equivalent{\rm \,:}
\begin{enumerate}
\item[{\rm (i)}] $Q$ is a standard parameter ideal for $M${\rm \,;}
\item[{\rm (ii)}] $e_1(Q,M) = -\sum_{i=1}^{r-1} \binom{r-2}{i-1} h^i(M)$.
\end{enumerate}
\end{Theorem}

\begin{proof} We have only to show the implication (ii) $\Rightarrow$ (i). To do this, we may assume that the residue class field
$\RR/\fkm$ of $\RR$ is infinite. We write $Q = (x_1, x_2, \dots, x_r)$, where each
$x_j$ is superficial for $M$ with respect to $Q$. Remember that by Lemma \ref{4.5}   $Q \H_\fkm^i
(M) = (0)$ for all $i \ne r$. Hence $Q$ is
standard, if $r=2$ (\cite[Corollary 3.7]{T}).

Assume that $r \geq 3$
and that our assertion holds true for $r-1$. Let $1 \le j \le r$ be an integer. We put $N=M /x_j M$, $\overline{M} = N/\H_\fkm^0(N)~(=M/\mathrm{U}_M(x_j))$, and $Q_j = (x_i \mid 1 \le i \le r, i \ne j)$. Then
$\H_\fkm^i(N) \cong \H_\fkm^i(\overline{M})$ for all $i \geq 1$. On
the other hand, since $x_j \H_\fkm^i(M) = (0)$ for $i \ne r$ and
$x_j$ is $M$-regular, for each $0 \leq i \leq r-2$ we have the short
exact sequence
\[0 \to \H_\fkm^i(M) \to
\H_\fkm^i(N) \to \H_\fkm^{i+1} (M) \to 0\]
 of local cohomology modules. Hence
${\mathbb I}(M) = {\mathbb I} (N)$ and
\begin{eqnarray*}
e_1(Q,M) = e_1(Q_j,N) & = & e_1(Q_j,\overline {M}) \\ &\geq&
- \sum_{i=1}^{r-2} \binom{r-3}{i-1} h^i(\overline{M}) \\ & = & -
\sum_{i=1}^{r-2} \binom{r-3}{i-1} h^i(N) \\ & = & - \sum_{i=1}^{r-2}
\binom{r-3}{i-1} [h^i(M) + h^{i+1} (M)] \\ & = & - \sum_{i=1}^{r-1}
\binom{r-2}{i-1} h^i(M) \\ & = & e_1(Q,M),
\end{eqnarray*}
so that  the equality
\[ e_1(Q_j,\overline {M}) = - \sum_{i=1}^{r-2}
\binom{r-3}{i-1} h^i(\overline{M})\] holds for the parameter
ideal ${Q_j}$ for the generalized Cohen--Macaulay $\RR$--module
$\overline{M} = M/\mathrm{U}_M(x_j)$. Thus  by the hypothesis of induction on $r=\dim_{\RR}M$, $Q_j$ is a standard parameter ideal for $M/\mathrm{U}_M(x_j)$ for every $1 \leq j \leq r$. Hence
$\mathrm{U}_M (x_1) \cap QM = x_1 M$ by Proposition \ref{4.6}. Thus $Q_1$
is standard parameter ideal for $M/x_1 M$ (\cite[Corollary 2.3]{T}).
Therefore $Q$ is a standard parameter ideal for $M$, because ${\mathbb
I}(M) = {\mathbb I} (M/x_1M)$ (\cite[Corollary 2.4]{T}).
\end{proof}

We are now ready to prove the main result of this section.

\begin{Theorem}\label{4.1}
Let $(\RR,\m)$ be a Noetherian local ring and $M$ an unmixed
$\RR$--module with $r= \dim_{\RR} M \geq 2$.  Then the following
conditions are equivalent{\rm \,:}
\begin{enumerate}
\item[{\rm (i)}] $M$ is a Buchsbaum $\RR$--module{\rm \,;}
\item[{\rm (ii)}] The
first Hilbert coefficient $e_1(Q,M)$ of $M$ is constant and independent of the choice of parameter ideals $Q$ for $M$.
\end{enumerate}
When this is the case, one has the equality
\[e_1(Q,M) =
-\sum_{i=1}^{r-1} \binom{r-2}{i-1}h^i(M)\] for every parameter ideal
$Q$ for $M$, where $h^i(M) = \lambda (\H_\fkm^i(M))$ for each $1 \leq i \leq r-1$.
\end{Theorem}

\begin{proof} (i) $\Rightarrow$ (ii) This is due to Schenzel \cite{Sch1}.

\medskip

\noindent (ii) $\Rightarrow$ (i)  Since $\sharp{\Lambda(M)} = 1$, by Proposition \ref{flc} $M$ is
a generalized Cohen--Macaulay $\RR$--module. Hence
$\Lambda (M) = \{-\sum_{i=1}^{r-1}\binom{r-2}{i-1}h^i(M)\}$ by
\cite[Korollar 3.2]{Sch1}, so that by Theorem \ref{4.7} every
parameter ideal $Q$ for $M$ is standard. Thus $M$ is,  by definition, a Buchsbaum
$\RR$--module (\cite{SV}).

\medskip

\noindent See \cite{Sch1} for the last assertion.
\end{proof}

\medskip

We now in a position to conclude this section with a characterization of $\RR$--modules
possessing $\sharp{\Lambda(M)} =1$.

\medskip

\begin{Theorem}
Let $(\RR,\m)$ be a Noetherian local ring and $M$  a finitely
generated $\RR$--module with $r = \dim_{\RR} M \geq 2$. Let $U=\mathrm{U}_{\widehat{M}}(0)$ be
the unmixed component of $(0)$ in the $\fkm$-adic completion
$\widehat{M}$ of $M$. Then the following
 conditions are equivalent{\rm \,:}
\begin{enumerate}
\item[{\rm (i)}]
$\sharp{\Lambda(M)} = 1${\rm \,;}
\item[{\rm (ii)}] 
$\widehat{M}/U$ is a Buchsbaum $\widehat{\RR}$--module and $\dim_{\widehat{\RR}} U \leq r-2$.
\end{enumerate}
When this is the case, one has the equality $$e_1(Q,M) =
-\sum_{i=1}^{r-1} \binom{r-2}{i-1} h^i(\widehat{M}/U)$$ for every
parameter ideal $Q$ for $M$.
\end{Theorem}

\begin{proof} We may assume $\RR$ is complete.

\noindent (i) $\Rightarrow$ (ii) Since $\sharp{\Lambda(M)}=1$, $\dim_{\RR} U \leq
r-2$ by
Proposition \ref{e1finite}.  We get $\sharp{\Lambda(M/U)}=1$ by Remark \ref{e1geqt}, so that  by Theorem \ref{4.1} $M/U$ is a
Buchsbaum $\RR$-module.

\medskip

\noindent (ii) $\Rightarrow$ (i) We get by Theorem \ref{4.1} that
$\sharp{\Lambda(M/U)}=1$ and hence $\sharp{\Lambda(M)}=1$ by
Lemma \ref{comparison}.

\medskip

\noindent See Theorem \ref{4.1} for the last assertion.
\end{proof}

\section{Homological degrees}

In this section we deal with the variation of the extended degree
function $\hdeg$ (\cite{DGV, hdeg}), labeled $\hdeg_I$ (see
\cite{Linh},
\cite[p. 142]{icbook}). We recall the basic properties of these
functions. These techniques and their relationships to $e_1(I)$ have
been mentioned in \cite{chern1} but the treatment here is more focused. It will lead to sharper bounds in the case of $e_{1}(I, M)$.

\subsubsection*{Cohomological degrees}
Let $\RR$ be a Noetherian local ring with maximal ideal $\m$ and infinite residue class field. Let
$\mathcal{M}(\RR)$ denote the category of finitely generated $\RR$-modules and let $I$ be an $\m$--primary ideal of $\RR$.
Then one has the following extension of the classical multiplicity.

\begin{definition}{\rm A {\em cohomological
degree}, or {\em extended
multiplicity function} with respect to $I$,
 is a function \label{Degnu}
\[\Deg(\cdot): {\mathcal M}(\RR) \to {\mathbb N}\]
that satisfies the following conditions. Let $M \in \mathcal{M}(\RR)$.
\begin{itemize}
\item[\rm {(a)}]  If $L = \Gamma_{\mathfrak m}(M)$ is the $\RR$-submodule of
elements of $M$ that are annihilated by a
power of the maximal ideal $\m$ and $\overline{M} = M/L$, then
\[ \Deg(M) = \Deg(\overline{M}) + \lambda(L). \]
\item[\rm {(b)}] (Bertini's rule)  If $M$ has positive depth, then
\[ \Deg(M) \geq \Deg(M/hM) \] for every generic hyperplane section  $h\in I \setminus  \mathfrak{m}I$.
\item[{\rm (c)}] (The calibration rule) If $M$ is a Cohen-Macaulay $\RR$--module, then
\[\Deg(M) = \deg(M), \]
where $\deg(M)=e_0(I,M)$ is the Samuel multiplicity of $M$ with respect to $I$.
\end{itemize}
}\end{definition}

\medskip

 The existence of cohomological degrees in arbitrary
dimensions was established in \cite{hdeg}. Let us formulate it for the case where the ring $\RR$ is complete.
The use of the more general Samuel multiplicities was introduced in \cite{Linh}.
When precision demands,   we denote the degree and homological degree 
  functions associated to the $\m$-primary ideal $I$
   by $\deg_I$ and $\hdeg_I$, respectively.

\medskip

For the rest of this section suppose that $\RR$ is complete. For each finitely generated $\RR$--module $M$ and $j \in \Bbb Z$ let $$M_j = \Hom_{\RR}(\rmH_\m^j(M), E),$$ where $E = \mathrm{E}_{\RR}(\RR/\m)$ denotes the injective envelope of the residue class field. Then, thanks to the local duality theorem, one gets $\dim_{\RR} M_j \le j$ for all $j \in \Bbb Z$.

\begin{definition}\label{hdegdef}{\rm Let $M$ be a finitely generated $\RR$-module with $r = \dim_{\RR}M > 0$. Then the {\em homological degree}\index{homological degree}\index{hdeg,
the homological degree} of $M$\label{hdegnu} is the integer
\[
\hdeg (M) = \deg(M) + \sum\limits_{j = 0}^{r-1} {{r-1}\choose{j}}\cdot  \hdeg (M_j).
\]
}\end{definition}

We call attention to the fact (see \cite{hdeg} for details) that the
notion of generic  hyperplane section used for $\hdeg(M)$ are
superficial elements for $M$ and for all $M_j$, but also for the
iterated ones of these modules (there are only a finite
number of them).

\medskip

 We will employ $\hdeg$ to derive lower bounds for
$e_1(I,M)$.

\subsubsection*{Homological torsion}\index{homological torsion}
There are
other combinatorial expressions of the terms $\hdeg_I(M_j)$ that
behave well under hyperplane sections.

\begin{definition}\label{6.3} {\rm
Let $M$ be an $\RR$-module with $r=\dim_{\RR}M \ge 2$. For each integer $1 \le  i \le r-1$ we put
\[ \TT_I^{(i)}(M)
=
  \sum\limits_{j=1}^{r-i} {{r-i-1}\choose{j-1}}\cdot
 \hdeg_I(M_j).\]
Hence
\[ \hdeg_I(M)> \TT_I^{(1)}(M) \geq \TT_I^{(2)}(M) \geq \cdots \geq \TT_I^{(r-1)}(M).\]
If $M$ is a generalized Cohen-Macaulay $\RR$--module, then
\[ \TT_I^{(i)}(M) = \sum_{j=1}^{r-i}\binom{r-i - 1}{j-1}\lambda(\rmH^j_{\mathfrak{m}}(M))\]
which is independent of $I$.
}
\end{definition}

We then have  the following.

\begin{Theorem}\label{torsionhdeg}{\rm {\cite[Theorem 2.13]{hdeg}}} Let $M$ be a finitely generated $\RR$--module with $r = \dim_\RR M$ and let
$h$ be a generic hyperplane section. Then
 $\TT_I^{(i)}(M/hM)\leq \TT_I^{(i)}(M)$ for all $1 \le i \le r-2 $.
\end{Theorem}

We now turn this into a uniform bound for the first Hilbert coefficient of
a module $M$ relative to an ideal $I$ generated by a system of parameters of $M$. We note that there are 
general bounds for all Hilbert coefficients $e_i(I,M)$ for arbitrary $\m$-primary ideals $I$ (\cite{RTV}).
Those developed here have a more specialized character and hold only for $e_1(I,M)$ and parameter ideals $I$.

\begin{Theorem} \label{Degreddim2}
 Let $M$ be a finitely generated $\RR$--module with $\dim_\RR M = \dim \RR \geq 2$ and let $Q$ be a parameter ideal of $\RR$. Then
\[ -e_{1}(Q, M) \leq \TT_{Q}^{(1)}(M).  \]
\end{Theorem}

\demo Let $d = \dim \RR$ and let $h \in Q \setminus \m Q$ be a generic hyperplane section used for $\hdeg_{Q}(M)$. Since
\[ -e_{1}(Q, M) = -e_{1}(Q, M/\rmH^{0}_{\m}(M)) \;\; \mbox{\rm and} \;\; \TT_{Q}^{(1)}(M/\rmH^{0}_{\m}(M)) \leq \TT_{Q}^{(1)}(M),  \]
replacing $M$ with $M/\rmH^{0}_{\m}(M)$ if necessary, we may assume $\depth_\RR M \ge 1$. We may also assume that $h$ is superficial for $M$ and for all $M_j$~($0 \le j \le d-1)$ with respect to $Q$. Hence $h$ is $M$--regular and $\l (M_1/hM_1) < \infty$ (remember that $\dim_\RR M_1 \le 1$).  Suppose $d = 2$. Then $\TT_{Q}^{(1)}(M) = \hdeg_{Q}(M_1)$ and $- e_{1}(Q, M) = \l((0):_{{\rmH}_\m^1(M)} h)$ by Proposition \ref{genhs} (d). On the other hand, from the exact sequence
\[ 0 \longrightarrow M \stackrel{h}{\longrightarrow} M \longrightarrow M/hM \longrightarrow 0 \]
of $\RR$--modules, we obtain the exact sequence
\[ 0 \to (0):_{{\rmH}_\m^1(M)} h \to \rmH_\m^1(M) \overset{h}{\to} \rmH_\m^1(M). \]
Then, taking the Matlis dual, we have an epimorphism
$$M_1/hM_1 \to \Hom_{\RR}((0):_{{\rmH}_\m^1(M)} h, E) \to 0,$$ so that
$$\l ((0):_{{\rmH}_\m^1(M)} h) = \hdeg_Q(\Hom_{\RR}((0):_{{\rmH}_\m^1(M)} h, E))\le \hdeg_Q(M_1/hM_1) \le \hdeg_Q(M_1)$$ by Theorem \ref{torsionhdeg}. Thus $-e_1(Q,M) \le \TT_Q^{(1)}(M)$. If $d \geq 3$, then we get
\[\TT_{Q/(h)}^{(1)}(M/hM) \leq \TT_{Q}^{(1)}(M). \]
Hence the result follows, since  $-e_{1}(Q, M) = -e_{1}(Q/(h), M/hM)$ by Proposition \ref{genhs} (a). \QED

\medskip

When the module is generalized Cohen-Macaulay we recover the bound discussed at the beginning of Section 4.

\medskip

\begin{Corollary}\label{6.7}
If $M$ is a generalized
Cohen-Macaulay $\RR$--module with $\dim_{\RR}M \ge 1$, then the Hilbert coefficients $e_1(Q,M)$
are bounded for all parameter ideals $Q$ for $M$.
\end{Corollary}

\begin{proof} Passing to the ring $\RR/[(0):_{\RR}M]$, we may assume that $\dim \RR = \dim_{\RR}M$ and that $Q$ is a parameter ideal of $\RR$. Then $e_1(Q,M)\leq 0$ by Corollary \ref{e1par}. We get by Theorem~\ref{Degreddim2} $-e_1(Q,M)\leq \TT_Q^{(1)}(M)$, while $\TT_Q^{(1)}(M) = \sum_{j = 1}^{d-1}\binom{d-2}{j-1}\l (\rmH_\m^j(M))$ is independent of the choice of $Q$. Hence the result.
\end{proof}

\medskip

\begin{Corollary}\label{lambdabarq}
Suppose that $\dim_{\RR} M \ge 1$.  Then the set
\[ \{e_1(Q, M) \mid Q~ \mbox{\rm are~parameter~ideals~of~$M$~with the same integral
closure}\}\]
is finite.
\end{Corollary}

\begin{proof} For each parameter ideal $Q$ of $M$ we get $e_1(Q,M)\leq 0$, while Theorem \ref{Degreddim2} asserts that
$e_1(Q,M) \geq - \TT_Q^{(1)}(M)$. Hence the result follows, because $\TT_Q^{(1)}(M)$ depends only
on
$\bar{Q}$, the integral closure of $Q$.
 \end{proof}

\medskip

\section{Euler characteristics and Hilbert characteristics}

The relationship between partial Euler characteristics and
superficial elements make for a straightforward comparison with
extended degree functions. Unless otherwise specified, throughout it is assumed that $\RR$ is a
Noetherian complete local ring with infinite residue class field.
We will  prove that Euler characteristics can be uniformly bounded by homological degrees.
The basic tool is the following observation, which is found   in the proof of \cite[Theorem 4.6.10 (a)]{BH}.

\begin{Proposition}\label{basicchi1} Let $M$ be a finitely generated $\RR$-module with $r = \dim_{\RR}M \ge 2$. Let $\xx=\{x_1, x_2, \ldots, x_r\}$
be a system of parameters for $M$ and set $\xx'=\{x_2, \ldots, x_r\}$.
Then
\[ \chi_1(\xx;M)= \chi_1(\xx';M/x_1M)+ \chi_1(\xx'; 0:_M{x_1}).\]

\end{Proposition}

\begin{Theorem}\label{chi1hdeg}
Let $M$ be a finitely generated $\RR$-module with $ \dim_{\RR}M = \dim \RR=d \ge 1$. Then for every system  $\xx = \{x_1, x_2, \ldots, x_d \}$ of
parameters of $\RR$, one has
\[ \chi_1(\xx;M)\leq \hdeg_{Q}(M)-\deg_{Q}(M),\]
where $Q = (\xx)$.
\end{Theorem}

\begin{proof} As $\l (M/QM) = \chi_1(\xx;M) + \deg_{Q}(M)$, we have only to show $\l (M/QM) \leq \hdeg_{Q}(M)$.
Let $h \in Q \setminus \m Q$ be a generic hyperplane section used for $\hdeg_Q(M)$. Then, since $\l (M/QM) = \l ((M/hM)/Q{\cdot}(M/hM))$ and $\hdeg_{Q/(h)}(M/hM) \le \hdeg_Q(M)$, by induction on $d$ we may assume $d=1$. When $d = 1$, $\chi_1(\xx;M)=\lambda(0:_Mx_1)$ and hence $\l (M/QM) = \chi_1(\xx;M) + \deg_Q(M) \le \lambda(\rmH_{\m}^0(M)) +\deg_{Q}(M) = \hdeg_{Q}(M)$, as wanted.
\end{proof}

\medskip

\begin{Corollary} \label{chi1intclos} Suppose that $\dim_{\RR} M \ge 1$. Then for every primary ideal $I$ of $M$, the set
\[ \Xi_I(M)=\{\chi_1(\xx, M)\mid \xx~\mbox{\rm ~are~ systems~of~parameters~of~M~with $\bar{(\xx)}=\overline{I}$}\}\]
is finite.
\end{Corollary}

\begin{proof} Both $\hdeg_{Q}(M)$ and $\deg_{Q}(M)$ depend
only on the integral closure $\overline{I} = \bar{Q}$ of $Q = (\xx)$.
\end{proof}

\medskip

\begin{Definition}\label{Hilbertchar}{\rm 
Let  $\RR$ be a Noetherian local ring and $M$  a finitely generated
$\RR$--module with $r = \dim_{\RR}M \ge 1$. For each system $\xx = \{x_1, x_2, \ldots, x_r \}$ of  parameters of $M$, the  {\em Hilbert characteristic  of $M$ with respect to $Q=(\xx)$} is defined to be
\[ \hh(\xx;M)= \sum_{i=0}^r (-1)^i e_i(Q, M).\]
}\end{Definition}

The following proposition shows that the Hilbert characteristic can be characterized as a {\em quasi-cohomological degree} for $M$.

\begin{Proposition}{\rm 
Let  $(\RR, \m)$ be a Noetherian local ring and $M$  a finitely generated $\RR$--module with $r = \dim_{\RR}M \ge 1$. Let  $\xx = \{x_1, x_2, \ldots, x_r \}$ be a system of parameters of $M$ and a $d$--sequence for $M$. Then the Hilbert characteristic of $M$ with respect to $\xx$ satisfies the following.
\begin{enumerate}
\item[$(\rm a)$] Suppose that $x_1$ is a superficial element for $M$ and  $\depth_{\RR} M\geq 1$. Then 
\[ \hh(\xx;M)=\hh(\xx';M/x_1M),\]
where $\xx'=\{x_2, \ldots, x_r\}$.

\item[$(\rm b)$] Let $M_0=\H_{\m}^0(M)$ and $M'=M/M_0$. Then 
\[ \hh(\xx;M)=\hh(\xx;M') + \lambda(M_0). \]
\end{enumerate}
}\end{Proposition}

\demo Let $Q=(\xx)$.  Recall that, by \cite[Proposition 3.4]{GO2}, we have ${\ds (-1)^{r} e_{r}(Q, M) = \l( H^{0}_{\m}(M))}$.

\medskip

\noindent (a) We may assume that $x_{1}$ is $M$--regular. By Proposition~\ref{genhs}, we obtain
\[ \hh(\xx ; M ) = \sum_{i=0}^{r-1} (-1)^i e_i(Q, M)  + (-1)^{r} e_{r}(Q, M) = \sum_{i=0}^{r-1} (-1)^i e_i(\xx', M/x_{1}M)  = \hh(\xx' ; M/x_{1}M). \]

\medskip

\noindent (b) By applying Proposition~\ref{genhs}-(b) to the exact sequence $0 \rar M_{0} \rar M \rar M' \rar 0$, we get
\[ e_{i}(Q, M) = e_{i}(Q, M') \quad \mbox{\rm for all} \; 0 \leq i \leq r-1.\]
Note that $(-1)^{r}e_{r}(Q, M') = \l( H^{0}_{\m}(M'))=0$. Hence
\[ \begin{array}{lll}
{\ds \hh(\xx; M) =   \sum_{i=0}^r (-1)^i e_i(Q, M)}  &=& {\ds  \sum_{i=0}^{r-1} (-1)^i e_i(Q, M) + \l( M_0) } \\ && \\
                  &= & {\ds \sum_{i=0}^{r-1} (-1)^i e_i(Q, M') + \l( M_0) } \\ && \\
                  &= & {\ds \sum_{i=0}^{r} (-1)^i e_i(Q, M') + \l( M_0) } \\ && \\
                  &= & {\ds \hh(\xx; M') + \l(M_0). } \\
\end{array}
\]
\QED

\medskip

\begin{Proposition}
Let  $(\RR, \m)$ be a Noetherian local ring and $M$  a finitely generated $\RR$--module with $r = \dim_{\RR}M \ge 1$. Let  $\xx = \{x_1, x_2, \ldots, x_r \}$ be a system of parameters of $M$ and a $d$--sequence for $M$. Let $Q=(\xx)$. 
Then 
\[ \hh(\xx;M)=\l(M/QM).\]
In particular, $\hh(\xx;M)\geq
e_0(Q, M)$ with equality if and only if $M$ is  Cohen--Macaulay. 
\end{Proposition}

\demo Using \cite[Theorem 3.7]{chern7}, one can prove that
\[ (-1)^i e_i(Q, M)= \chi_1(x_{1}, \ldots, x_{r-i}, x_{r-i+1}; M)- \chi_1(x_{1}, \ldots, x_{r-i}; M) \geq 0\] 
for all $1 \leq i \leq r$. Therefore
\[\begin{array}{lll}
\hh(\xx; M) &=& {\ds  e_0(Q, M) + \sum_{i=1}^r (-1)^i e_i(Q, M) } \\ && \\
                  &=& {\ds  e_0(Q, M) + \sum_{i=1}^r (  \chi_1(x_{1}, \ldots, x_{r-i}, x_{r-i+1}; M)- \chi_1(x_{1}, \ldots, x_{r-i}; M)  ) } \\ && \\
                  &=& {\ds  \chi_{0}(\xx; M) + \chi_{1}(\xx; M) } \\ && \\
                  &=& {\ds \l(M/QM) }.\\
\end{array}\]
\QED

\medskip

\begin{Corollary}\label{betti} Let $\xx$ be a system of parameters of $M$ which is a $d$--sequence for $M$. Suppose that $\xx \in \m \setminus \m^2$. Then
the Betti numbers $\beta_i^{\RR}(M)$ satisfy
\[ \beta_i^{\RR}(M)\leq \lambda(M/(\xx)M)\cdot \beta_i^{\RR}(k).\]
\end{Corollary}

\demo
It follows from the argument of \cite[Theorem 2.94]{icbook},
 where we use the properties of $\hh(\xx;M)$  in the
induction part.
\QED

\medskip

\begin{Remark}{\rm 
Note that the condition $\xx \in \m \setminus \m^2$ in Corollary~\ref{betti} is needed in the induction argument which requires the inequality of Betti numbers $\beta_{i}^{\RR/(x_{1})}(k) \leq \beta_{i}^{\RR}(k)$(\cite[Corollary 3.4.2]{GL}).
}\end{Remark}

\section{Buchsbaum-Rim coefficients}

In this section let us note another set of related  questions, concerned about
the
vanishing and the negativity of the
Buchsbaum-Rim coefficients of modules.

\medskip

Let $\RR$ be a
 Noetherian local ring with maximal ideal $\m$ and $d = \dim \RR \ge 1$.
The Buchsbaum-Rim multiplicity (\cite{BR}) arises in the context of an embedding
\[ 0 \rar E \lar F=\RR^s \lar C \rar 0 \]
of $\RR$--modules, where $E \subseteq \m F$ and $C$ has finite length.
 Let
\[ \varphi: \RR^m \lar F=\RR^s\]
be an $\RR$--linear map represented by a matrix with entries in $\m$ such that $\operatorname{Im} \varphi = E$. We then have a  homomorphism
\[ {\mathcal S}(\varphi): {\mathcal S}(\RR^m) \lar {\mathcal S}(\RR^s)\] of symmetric algebras,
whose image is the Rees algebra $\Rees(E)$ of $E$, and whose cokernel we denote by
$C(\varphi)$. Hence
\[
0 \rar \Rees(E) \lar {\mathcal S}(\RR^s)=\RR[T_1,T_2,  \ldots, T_s] \lar C(\varphi)\rar 0. \label{br1}
\]
This exact sequence (with a different notation) is studied in
\cite{BR} in great detail. Of significance for us is the fact that
$C(\varphi)$, with the grading induced by the homogeneous homomorphism
${ S}(\varphi)$, has components of finite length, for which we have the following.
Let $E^n = [\Rees (E)]_n$ and $F^n = [{\mathcal S}(F)]_n$ for $n \ge 0$, where $F = \RR^s$.

\begin{thm}  \label{degofBR}
$\lambda(F^n/E^n)$ is a polynomial in $n$ of degree $d+s-1$ for
$n\gg 0$$\mathrm{:}$
\[ \lambda(F^n/E^n) = \br(E) {{n+d+s-2}\choose{d+s-1}}-\br_1(E)
{{n+d+s-3}\choose{d+s-2}}
  + \textrm{\rm lower
terms}. \]
\end{thm}

This polynomial is
 called the Buchsbaum-Rim polynomial of
$E$\index{Buchsbaum-Rim coefficients of a module}.
The leading coefficient $\br(E)$ is the Buchsbaum-Rim multiplicity of $\varphi$; if the homomorphism $\varphi$  is understood, we shall simply denote it by $\br(E)$. This
number  is determined by
an Euler characteristic of the Buchsbaum-Rim complex (\cite{BR}).

\medskip

Assume now the residue class field of $\RR$  is infinite.
The minimal reductions $U$ of $E$ are generated
by $d+s-1$ elements. We refer to $U$ as a parameter module of $F$.
The corresponding coefficients are $\br(U)=\br(E)$
but $\br_1(U)\leq \br_1(E)$. It is  not clear what
the possible values of $\br_1(U)$ are, and in similarity to the case of
ideals, we can ask the following.

\begin{enumerate}
\item[$(\rm a)$] $\br_1(U)\le 0$?
\item[$(\rm b)$] Suppose that  $\RR$ is unmixed. Then is $\RR$ Cohen-Macaulay, if $\br_1(U) = 0$?
\item[$(\rm c)$] Are the values of $\br_1(U)$ 
bounded?
\item[$(\rm d)$] What happens in low dimensions?
\end{enumerate}

As for question (a), a surprising result of Hayasaka and Hyry shows
the negativity of $\br_1(U)$ in the following way. It gives an
eminent proof of Corollary \ref{e1par}.

\begin{thm}[{\cite[Theorem 1.1]{HaHy10}}]\label{HaHy}
$\lambda(F^n/U^n) \ge \br(U) {{n+d+s-2}\choose{d+s-1}}$ for all $n
\ge 0$. Hence \[\br_1(U) \le 0.\]
\end{thm}

\noindent They also proved that $\RR$ is a Cohen-Macaulay ring, once
$\lambda(F^n/U^n) = \br(U) {{n+d+s-2}\choose{d+s-1}}$ for some $n \ge 1$. When this is the case, one has the equality $\lambda(F^n/U^n) =
\br(U) {{n+d+s-2}\choose{d+s-1}}$ for all $n \ge 0$, whence $\br_1(U)
= 0$ (\cite[Theorem 3.4]{UV}).

\medskip

Note that question (c) is answered affirmatively for $s=1$ in
Corollary~\ref{lambdabarq}.

\medskip

We close this paper with the following.

\medskip

\begin{conj}\label{negbr1}{\rm Let $(\RR, \m)$ be a Noetherian   local
ring with $\dim \RR \ge 2$ and let $U\subseteq \m\RR^s$ be a parameter module of $\RR^s$~$(s> 0)$. Then
$\RR$ is a Cohen-Macaulay ring if and only if $\RR$ is
unmixed and $\br_1(U)= 0$.
}
\end{conj}

\end{document}